\documentclass[11pt,reqno,a4paper]{amsart}
\usepackage{amsmath,amsthm,amsfonts,amssymb,graphicx,psfrag,dsfont}

\usepackage{hyperref}
\usepackage[alphabetic,initials]{amsrefs}

\psfrag{M}{$M$}
\psfrag{B}{$B$}
\psfrag{Ppsi}{$P_\psi$}
\psfrag{R}{$\RR$}

\theoremstyle{plain}
\newtheorem*{mainthm}{Main Theorem}

\newtheorem{prop}{Proposition}

\newtheorem{lemma}[prop]{Lemma}

\theoremstyle{definition}

\newtheorem{remark}[prop]{Remark}

\theoremstyle{remark}

\newcommand{\ev}{\operatorname{ev}}

\newcommand{\ind}{\operatorname{ind}}

\newcommand{\muCZ}{\mu_{\operatorname{CZ}}}

\newcommand{\Symp}{\operatorname{Symp}}

\newcommand{\DD}{{\mathbb D}}

\newcommand{\RR}{{\mathbb R}}

\newcommand{\ZZ}{{\mathbb Z}}
\newcommand{\hH}{{\mathcal H}}
\newcommand{\jJ}{{\mathcal J}}

\newcommand{\mM}{{\mathcal M}}

\newcommand{\uU}{{\mathcal U}}

\newcommand{\p}{\partial}

\numberwithin{equation}{section}

\title[Open Books and Stable Hamiltonian Structures]{Open 
Book Decompositions and Stable Hamiltonian Structures}
\author{Chris Wendl}
\address{ETH Z\"urich \\ Departement Mathematik, HG G38.1 \\ 
R\"amistrasse 101 \\
8092 Z\"urich \\ 
Switzerland}
\email{wendl@math.ethz.ch}
\urladdr{http://www.math.ethz.ch/~wendl/}
\thanks{Research partially supported by an NSF Postdoctoral Fellowship
(DMS-0603500).}

\begin{document}

\begin{abstract}
We show that every open book decomposition of a contact $3$--manifold
can be represented (up to isotopy) by a smooth $\RR$--invariant family
of pseudoholomorphic curves on its symplectization with respect to a
suitable stable Hamiltonian structure.  In the planar case, this family
survives small perturbations, and thus gives a concrete construction
of a stable finite energy foliation that has been used in various 
applications to planar contact manifolds, including the Weinstein 
conjecture \cite{ACH} and the equivalence of strong and Stein fillability
\cite{Wendl:fillable}.
\end{abstract}

\maketitle

%\tableofcontents

\section{Introduction}
\label{sec:intro}

The subject of this note is a correspondence between open book
decompositions on contact manifolds and $J$--holomorphic curves in their
symplectizations.  We will assume
throughout that $(M,\xi)$ is a closed $3$--manifold with a positive,
cooriented contact structure.  An \emph{open book decomposition} of $M$
is a fibration 
$$
\pi : M \setminus B \to S^1,
$$
where $B \subset M$ is a link called the \emph{binding}, and the
fibers are called \emph{pages}: these are open surfaces whose closures
have boundary equal to~$B$.  An open book is called \emph{planar}
if the pages have genus zero, and it is said to \emph{support}
a contact structure $\xi$ if the latter can be written as $\ker\lambda$
for some contact form $\lambda$ (a \emph{Giroux form}) 
such that $d\lambda$ is positive on the
pages and $\lambda$ is positive on the binding (oriented as the boundary
of the pages).  In this case the Reeb vector field $X_\lambda$ defined by
$\lambda$ is transverse to the pages and parallel to the binding, so in
particular the binding is a union of periodic orbits.  A picture of a simple
open book on the tight $3$--sphere is shown in Figure~\ref{fig:openbook}.

\begin{figure}
\includegraphics{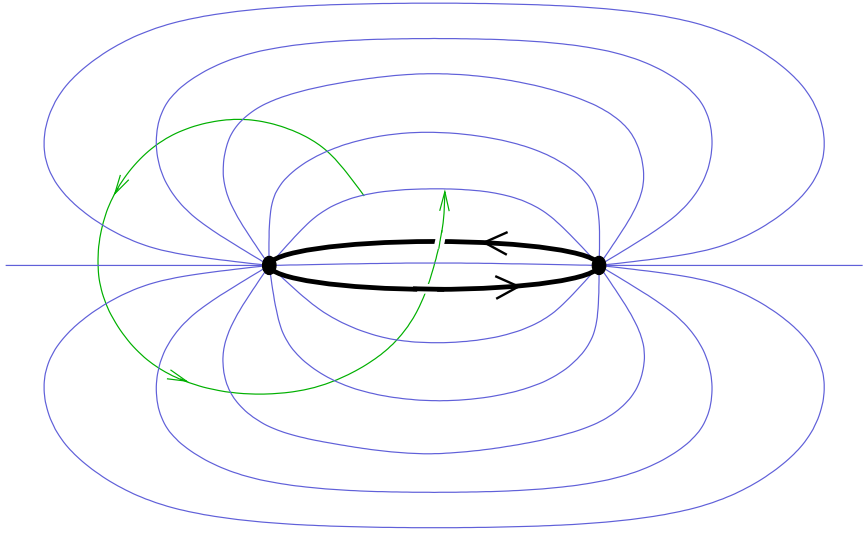}
\caption{\label{fig:openbook} 
An open book decomposition of the tight $3$--sphere with one binding orbit
and disk-like pages, which are transverse to the Reeb vector field.}
\end{figure}

We say that an almost complex
structure $J$ on $\RR\times M$ is \emph{compatible with $\lambda$} if
it is invariant under the natural $\RR$--action, maps the unit vector
in the $\RR$--direction to $X_\lambda$ and restricts to $\xi$ as a complex
structure compatible with $d\lambda|_{\xi}$.  One then
considers $J$--holomorphic curves 
$$
u : \dot{\Sigma} \to \RR \times M,
$$
where the domain is a closed Riemann surface with finitely many punctures,
and $u$ satisfies a finite energy condition (see \cite{SFTcompactness}),
so that it has ``asymptotically cylindrical'' behavior at the punctures,
approaching closed orbits of $X_\lambda$ at $\{\pm\infty\} \times M$.
Note that whenever the projection of $u$ into $M$ is embedded, it is also
transverse to $X_\lambda$, a property
that is shared by the pages of supporting open books with their Giroux forms.
Thus it is natural to ask whether the pages of an open book can in general
be presented as projections of holomorphic curves: such a family of
holomorphic curves is referred to as a \emph{holomorphic open book}, and is
a special case of a \emph{finite energy foliation}
(see \cites{HWZ:foliations,Wendl:OTfol}).
We refer to \cite{Etnyre:lectures} for further details on the rich 
relationship between open books and contact structures, and
\cite{Hofer:real} for some applications of holomorphic curves in this
context to dynamics.  

Our main goal is to prove the following.

\begin{mainthm}
%\label{thm:main}
Suppose $\pi : M \setminus B \to S^1$ is a planar open
book decomposition on $M$ that supports~$\xi$.  Then after an isotopy of $\pi$,
it admits a nondegenerate Giroux form $\lambda$ with a compatible
almost complex structure $J$ on the symplectization
$\RR\times M$, and a smooth $2$--dimensional $\RR$--invariant family of
embedded, finite energy $J$--holomorphic curves in $\RR\times M$ 
with index~$2$, whose projections to $M$ give an $S^1$--family of 
embeddings parametrizing the pages of~$\pi$.
\end{mainthm}

\begin{remark}
\label{remark:smallPeriods}
It will be clear from the construction that one can choose the Giroux
form $\lambda$ in this theorem so that the binding orbits have
arbitrarily small periods compared with all other Reeb orbits in~$M$.
This assumption is sometimes useful for compactness arguments, and
is exploited e.g.~in \cite{AlbersBramhamWendl}.
\end{remark}

This result has been used in the literature for various 
applications, including Abbas-Cieliebak-Hofer's
proof of the Weinstein conjecture for planar contact manifolds \cite{ACH}, 
and the author's theorem that
strong symplectic fillings of such manifolds are always blowups of 
Stein fillings
\cite{Wendl:fillable}.  A construction of holomorphic open books was sketched 
in \cite{ACH} without many details.  The construction explained below is 
based on a completely different idea,
and has the advantage of producing a (usually non-stable)
finite energy foliation out of \emph{any} open book, with arbitrary genus.
The catch is that this construction requires a choice of $J$ which is not
compatible with $\lambda$ in the sense described above, but is instead
compatible with a \emph{stable Hamiltonian structure}, which can be seen
as a limit of $\lambda$ as the contact structure degenerates to a confoliation.
The idea is then to recover the contact case by a perturbation argument, 
but for analytical reasons, this can only be done with a planar open book.

The trouble with the non-planar case is that
holomorphic curves of higher genus with the desired intersection
theoretic properties never have positive index, and thus generically
cannot exist.  This problem has an analogue in the study of closed symplectic
$4$--manifolds, namely in McDuff's classification \cite{McDuff:rationalRuled}
of manifolds that admit nonnegative symplectic spheres---there is no
corresponding result for higher genus symplectic surfaces because the 
dimension of the moduli space of higher genus holomorphic curves is
generally too small.  In the contact setting, a potential remedy was proposed
by Hofer in \cite{Hofer:real}, who
suggested considering a more general elliptic problem in which a harmonic
$1$--form is introduced to raise the index.
The study of this problem is a large project in progress by C.~Abbas
\cite{Abbas:openbook} and Abbas-Hofer-Lisi \cite{AbbasHoferLisi}, in which
punctured holomorphic curves of genus zero 
are treated as an easy special case: this would be a necessary ingredient
to generalize the approach in \cite{ACH} to the Weinstein conjecture
in dimension three.\footnote{In the mean time, Taubes \cite{Taubes:weinstein}
has produced a proof of the Weinstein conjecture in dimension three
based on Seiberg-Witten theory.  It is generally believed that a 
proof based on holomorphic curves should also be possible, but none
has yet appeared.}
For other applications however, it is already 
helpful to know that any open book can be viewed as a family of 
$J$--holomorphic curves for some non-generic choice of~$J$.
This idea is exploited for instance
in \cite{Wendl:openbook2} to compute certain algebraic invariants of
contact manifolds based on holomorphic curves, and
in \cite{Wendl:fiberSums} to define previously unknown obstructions
to symplectic filling.

Our construction rests on the notion of an \emph{abstract open book},
which is defined by the data $(P,\psi)$, where $P$ is a compact
oriented surface with boundary representing the \emph{page}, and $\psi$ is a
diffeomorphism that fixes the boundary, called the \emph{monodromy
map}.  Without loss of generality, we can assume that $\psi$ is the
identity in a neighborhood of $\p P$.
Let $P_\psi$ denote the \emph{mapping torus} of
$\psi$, which is the smooth $3$--manifold with boundary,
$$
P_\psi = (\RR \times P) / \sim
$$
where $(t+1,p) \sim (t,\psi(p))$.  This comes with a natural
fibration $P_\psi \to S^1 := \RR / \ZZ$, so that
the tangent spaces to~$P$ define a $2$--plane distribution
in~$T P_\psi$, called the \emph{vertical distribution}.

\begin{prop}
\label{prop:SHS}
Given an abstract open book $(P,\psi)$, let $M$ denote the closed
$3$--manifold obtained by gluing solid tori to $P_\psi$ so that
$(P,\psi)$ defines an open book decomposition of~$M$.  Then the
vertical distribution on $P_\psi$ can be extended to $M$ as a
confoliation $\xi_0$, such that a $C^\infty$--small perturbation of
$\xi_0$ defines a contact structure $\xi_\epsilon$ supported by the
open book, and each is compatible with stable Hamiltonian structures
$\hH_0 = (\xi_0,X_0,\omega_0)$ and $\hH_\epsilon = (\xi_\epsilon,
X_\epsilon,\omega_\epsilon)$ such that $\hH_\epsilon$ is
$C^\infty$--close to $\hH_0$.
\end{prop}

We will prove this via a concrete construction in the next section,
after recalling precisely what 
a stable Hamiltonian structure $\hH = (\xi,X,\omega)$ is, and how it defines 
a special class of almost complex structures on $\RR\times M$.
The next step is the following:

\begin{prop}
\label{prop:J0openbook}
Given $(P,\psi)$, $M$ and $\hH_0 = (\xi_0,X_0,\omega_0)$ as in
Proposition~\ref{prop:SHS}, there exists an almost complex structure 
$J_0$ compatible with $\hH_0$ such that the pages of the 
open book on~$M$ lift to 
embedded $J_0$--holomorphic curves in $\RR\times M$, with positive ends and 
index $2 - 2g$, where $g$ is the genus of~$P$.
\end{prop}

We will prove this in \S\ref{sec:FEF}, and
recall the definition of the \emph{index} of
a $J$--holomorphic curve and its significance.  We'll then show that 
the index~$2$ curves obtained in the case $g=0$ survive
as a holomorphic open book under the
small perturbation from $\hH_0$ to~$\hH_\epsilon$, thus proving
the main theorem.

\subsection*{Acknowledgments}
Thanks to Dietmar Salamon and Sam Lisi for helpful conversations.

\section{Stable Hamiltonian structures}
\label{sec:SHS}

Stable Hamiltonian structures were introduced in \cite{SFTcompactness}
(although the name came somewhat later, 
cf.~\cites{EliashbergKimPolterovich,Eliashberg:SFT})
as a generalized setting for holomorphic curves in symplectizations that
accomodates both contact geometry and Floer homology, among other things.
Given a closed oriented $3$--manifold $M$, a stable Hamiltonian structure
$\hH = (\xi,X,\omega)$ is defined by
\begin{enumerate}
\item
a coorientable $2$--plane distribution $\xi \subset TM$,
\item
a vector field $X$ (the \emph{Reeb vector field}) that is everywhere 
transverse to $\xi$ and has flow preserving $\xi$,
\item
a closed $2$--form $\omega$ (the \emph{taming form}) 
such that $\omega|_{\xi} > 0$ and $\iota_X\omega \equiv 0$.
\end{enumerate}
One can associate to $\hH = (\xi,X,\omega)$ the unique $1$--form $\lambda$
such that $\ker\lambda = \xi$ and $\lambda(X) \equiv 1$, which then
automatically satisfies $d\lambda(X,\cdot) \equiv 0$.  Moreover,
$\iota_X\omega \equiv 0$ implies that the flow of $X$ preserves not only
$\xi$ but also its symplectic structure defined by $\omega|_{\xi}$.
Now, given $\hH$,
the so-called \emph{symplectization} $\RR \times M$ inherits a 
natural splitting $T(\RR \times M) = \RR \oplus \RR X \oplus \xi$,
and we use this to define a special class of almost complex 
structures $\jJ(\hH)$ on $\RR\times M$, so that for every $J \in \jJ(\hH)$,
\begin{enumerate}
\item
$J$ is invariant under the natural $\RR$--action on $\RR\times M$,
\item
$J \p_a = X$, where $\p_a$ denotes the unit vector in the $\RR$--direction,
\item
$J(\xi) = \xi$ and $J|_{\xi}$ is compatible with the symplectic structure
$\omega|_{\xi}$.
\end{enumerate}
Notice that the definition of $\jJ(\hH)$ depends on $\omega$ only up to
the \emph{conformally} symplectic structure that it induces on~$\xi$.
Thus one can always replace $\hH = (\xi,X,\omega)$ by 
$\hH' = (\xi,X,f\omega)$ for any smooth
function $f : M \to (0,\infty)$ with $df \wedge \omega = 0$; then
$\jJ(\hH) = \jJ(\hH')$, and the notions of finite energy $J$--holomorphic
curves defined via $\hH$ and $\hH'$ coincide.

If the distribution $\xi$ in $\hH = (\xi,X,\omega)$ is a
contact structure, then any choice of $\lambda$ with $\ker\lambda = \xi$
defines $X$ uniquely:
in standard contact geometric terms, it is the Reeb vector field determined
by the contact form $\lambda$.  In this case one can also take $d\lambda$
as a natural taming form, though as mentioned above, it is not the 
\emph{only} choice.

Symplectic fibrations over $S^1$ provide another natural source of
stable Hamiltonian structures.  Suppose $\pi : M \to S^1$ is a locally trivial
symplectic fibration whose standard fiber is a symplectic surface
$(S,\sigma)$, possibly with boundary, and denote the coordinate on
the base by~$t$.  The vertical subspaces form an integrable distribution
$\xi \subset TM$, and any symplectic connection can be defined so that
parallel transport is the flow of a vector field $X$ on $M$ with
$\pi_*X = \p_t$.  There is then a unique $2$--form $\omega$ on $M$ such
that $\omega|_{\xi} = \sigma$ and $\omega(X,\cdot) \equiv 0$, and we claim
that $\hH := (\xi,X,\omega)$ is a stable Hamiltonian structure on~$M$.  One
only has to verify that $\omega$ is closed; to see this, identify a 
neighborhood of any point in $M$ with $(-\epsilon,\epsilon) \times S$
via a symplectic local trivialization and denote the real coordinate by~$t$.  
Then $X$ can be written on $(-\epsilon,\epsilon) \times S$ in the form
$$
X(t,p) = \p_t + V(t,p)
$$
for some $t$--dependent locally Hamiltonian vector field $V_t = V(t,\cdot)$
on $S$, and $\sigma$ defines a $2$--form on $(-\epsilon,\epsilon) \times S$ with
$\p_t \in \ker\sigma$.  One can then check that $\omega$ has the form
$\sigma + \iota_V\sigma \wedge dt$, which is closed because 
$\iota_{V_t}\sigma$ is a closed
$1$--form on $S$ for every~$t$.  An important special case of this
construction is the mapping torus $S_\varphi$ for a symplectomorphism
$\varphi \in \Symp(S,\sigma)$: then the Floer homology of $(S,\varphi)$
can be viewed as a special case of symplectic field theory on
$(S_\varphi,\hH)$.

We shall now prove Proposition~\ref{prop:SHS} by constructing a stable Hamiltonian
structure that combines both of the examples above.  The resulting distribution
$\xi$ will be a \emph{confoliation} (cf.~\cite{EliashbergThurston}), 
which means that the associated
$1$--form satisfies $\lambda \wedge d\lambda \ge 0$; it is a contact structure
wherever this inequality is strict, and is a symplectic fibration
everywhere else.

Suppose $(P,\psi)$ is an abstract open book, and $\phi : P_\psi \to S^1$ is
its mapping torus, regarded as a fibration over $S^1$, with the vertical 
distribution denoted by $\xi_0 \subset T P_\psi$.  
For some neighborhood $\uU$ of each component of $\p P$, choose
$\delta > 0$ small and identify $\uU$ with
$[1-\delta,1 + \delta) \times S^1$ by a diffeomorphism
$$
(\rho,\theta) : \uU \to [1-\delta,1+\delta) \times S^1
$$
such that $d\theta \wedge d\rho$ defines the positive orientation of~$P$.
We can assume without loss of generality that $\psi$ is the identity
on~$\uU$, so a neighborhood of each boundary component of $P_\psi$ now looks
like $S^1 \times [1-\delta,1+\delta) \times S^1$ with coordinates
$(\phi,\rho,\theta)$.

We will also use the symbols $(\theta,\rho,\phi)$ to denote coordinates
on the solid torus $S^1 \times \DD$, where $\theta$ is assigned to
the first factor and
$(\rho,\phi) \in [0,1] \times S^1$ are polar coordinates on the closed
unit disk $\DD \subset \RR^2$.  Then there is a closed manifold
$$
M := P_\psi \cup_{\p P_\psi} \left( \bigcup S^1 \times \DD \right)
$$
defined by gluing a copy of $S^1 \times \DD$ to each boundary component
of $P_\psi$, with attaching maps defined to be the identity in the coordinates
$(\theta,\rho,\phi) \in S^1 \times [1-\delta,1] \times S^1$.  
Denoting the union of all the loops $\{ \rho = 0\}$ by $B$, we now have
a natural fibration $\pi : M \setminus B \to S^1$
defined by the $\phi$--coordinate.

Define $\lambda_0|_{P_\psi} = d\phi$, so $\xi_0 = \ker\lambda_0$.
Then we can extend $\xi_0$ over $M$ as a confoliation
by defining $\lambda_0$ for $\rho < 1 + \delta$ as
$$
\lambda_0 = f(\rho)\ d\theta + g(\rho)\ d\phi
$$
for some pair of smooth functions $f,g : [0,1+\delta] \to \RR$ such that
\begin{enumerate}
\item
The path $\rho \mapsto (f(\rho),g(\rho)) \in \RR^2$ moves through the
first quadrant from $(f(0),g(0)) = (c,0)$ for some $c > 0$ to 
$(f(1 + \delta),g(1 + \delta)) = (0,1)$ and is constant
for $\rho \in [1-\delta,1+\delta]$.
\item
The function 
$$
D(\rho) := f(\rho) g'(\rho) - f'(\rho) g(\rho)
$$
is positive on $(0,1-\delta)$, and $g''(0) > 0$.
\item
There is a small number $\delta' > \delta$ such that
$g(\rho) = 1$ for all $\rho \in [1 - \delta',1+\delta)$.
\item
The maps $\DD \to \RR$ defined by $(\rho,\phi) \mapsto f(\rho)$ and 
$(\rho,\phi) \mapsto g(\rho) / \rho^2$ are smooth at the origin.
\end{enumerate}
The last condition requires $f'(0) = g'(0) = 0$, and it ensures that
$\lambda_0$ is well defined and smooth at the coordinate singularity 
$\rho = 0$.  The second guarantees that $\lambda_0$ is contact for
$\rho < 1 - \delta$, and the significance of the third can be
seen by computing the Reeb vector field in this region:
we find
\begin{equation}
\label{eqn:Reeb}
X_0(\theta,\rho,\phi) = \frac{g'(\rho)}{D(\rho)} \p_\theta -
\frac{f'(\rho)}{D(\rho)} \p_\phi,
\end{equation}
which is identically equal to $\p_\phi$ for $\rho \in [1 - \delta',
1 - \delta)$.

It follows from a fundamental theorem of Giroux \cite{Giroux:openbook}
that every
closed contact $3$--manifold is isomorphic to $(M,\xi_\epsilon)$,
where $\xi_\epsilon$ is a small perturbation
of a confoliation~$\xi_0$ as constructed above.  
Let us make this perturbation explicit:
following \cite{Etnyre:lectures}, choose a smooth $1$--form $\eta$ on $P$ 
such that $\eta = (2 - \rho)\ d\theta$ near the boundary and $d\eta > 0$
everywhere.  Then if $\tau : [0,1] \to [0,1]$ is a smooth function that
equals $0$ for $t$ on a neighborhood of~$0$ and $1$ for 
$t$ on a neighborhood of~$1$, we define a $1$--form on $[0,1] \times P$ by
$$
\alpha = \tau(\phi) \psi^*\eta + \left[ 1 - \tau(\phi) \right] \eta,
$$
where $\phi$ is now the coordinate on $[0,1]$.  This extends to $\RR \times P$
and then descends to a smooth $1$--form on $P_\psi$ such that
$\alpha = (2 - \rho)\ d\theta$ near $\p P_\psi$ and $d\alpha|_{\xi_0} > 0$.
Then for sufficiently small $\epsilon > 0$,
$$
\lambda_\epsilon := d\phi + \epsilon \alpha
$$
is a contact form on $P_\psi$: indeed, $d\lambda_\epsilon =
\epsilon \ d\alpha$ is positive on $\xi_0$, and thus also on the
$C^\infty$--close perturbation $\xi_\epsilon := \ker\lambda_\epsilon$.
In the region $\rho \in [1 - \delta,1 + \delta)$ near any component
of $\p P_\psi$, we now have $\lambda_\epsilon = \epsilon (2-\rho)\ d\theta
+ d\phi$, thus we can extend $\lambda_\epsilon$ to a contact form on $M$
close to $\lambda_0$ by choosing $C^\infty$--small perturbations 
$(f_\epsilon,g_\epsilon)$ of $(f,g)$ such that
\begin{enumerate}
\item
$(f_\epsilon(\rho),g_\epsilon(\rho)) = (\epsilon(2-\rho),1)$ for
$\rho \in [1 - \delta,1+\delta)$,
\item
$g_\epsilon(\rho) = 1$ and $f'_\epsilon(\rho) < 0$ 
for all $\rho \in [1 - \delta',1 + \delta)$,
\item
$(f_\epsilon(\rho),g_\epsilon(\rho)) = (f(\rho),g(\rho))$ for
$\rho \in [0,1 - \delta']$.
\end{enumerate}
Now if $X_\epsilon$ denotes the Reeb vector field determined by 
$\lambda_\epsilon$, we have $X_\epsilon \equiv X_0$ on 
$\{ \rho < 1 - \delta\}$; in particular this equals $\p_\phi$ for
$\rho \in [1 - \delta',1 - \delta)$.

\begin{lemma}
\label{lemma:ReebField}
$X_0$ extends over $M$ as the $C^\infty$--limit of
$X_\epsilon$ as $\epsilon \to 0$.
\end{lemma}
\begin{proof}
On $P_\psi$, the direction of $X_\epsilon$ is determined by
$\ker d\lambda_\epsilon = \ker d\alpha$, and is therefore independent of
$\epsilon$, so $X_\epsilon$ converges as $\epsilon \to 0$ to the unique
vector field that spans $\ker d\alpha$ and satisfies $d\phi(X_0) \equiv 1$.
In a neighborhood of $\p P_\psi$ this is simply $\p_\phi$, so it fits
together smoothly with \eqref{eqn:Reeb}.
\end{proof}

To complete the proof of Proposition~\ref{prop:SHS}, choose a smooth function
$h : [1 - \delta',1 - \delta] \to (0,\infty)$ which equals
$-f'(\rho)$ for $\rho$ near $1 - \delta'$ and $1$ for $\rho$ near
$1 - \delta$.  Then a taming form for $\xi_0$ and $X_0$ can be defined by
$$
\omega_0 = \begin{cases}
d\alpha & \text{ on $P_\psi$,}\\
h(\rho) \ d\theta \wedge d\rho & 
\text{ for $\rho \in [1 - \delta',1 - \delta)$,}\\
d\lambda_0 & \text{ for $\rho < 1 - \delta'$,}
\end{cases}
$$
making $\hH_0 := (\xi_0,X_0,\omega_0)$ into a stable Hamiltonian structure.
Since $X_\epsilon$ and $X_0$ are everywhere colinear and $\xi_\epsilon$
is assumed close to $\xi_0$, $\omega_0$ also furnishes a taming form
for $\xi_\epsilon$ and $X_\epsilon$, defining $\hH_\epsilon :=
(\xi_\epsilon,X_\epsilon,\omega_0)$.  Observe that $\omega_0 = F_\epsilon
\ d\lambda_\epsilon$ for a function $F_\epsilon$ with
$d F_\epsilon \wedge d\lambda_\epsilon = 0$.

\begin{remark}
The data $\hH_0 = (\xi_0,X_0,\omega_0)$ give $P_\psi$
the structure of a symplectic fibration, where the symplectic form on 
the fibers is $d\alpha|_{\xi_0}$, and $X_0$ defines a symplectic connection.
\end{remark}

\begin{remark}
The taming form $\omega_0$ will not play any role in the arguments to
follow, but it is important in further applications, 
cf.~\cite{Wendl:openbook2}.  In particular, one needs it to obtain
compactness results for a sequence of $J_\epsilon$--holomorphic curves
with $J_\epsilon \in \jJ(\hH_\epsilon)$ as $\epsilon \to 0$.
\end{remark}

\section{Finite energy foliations}
\label{sec:FEF}

Let us now apply Proposition~\ref{prop:SHS} to prove 
Proposition~\ref{prop:J0openbook} and the main theorem.  We begin by
choosing an appropriate $J_0 \in \jJ(\hH_0)$ and constructing
a foliation of $\RR \times M$ by $J_0$--holomorphic curves.  On $P_\psi$
this is easy: $\xi_0$ is tangent to the fibers and is preserved by any
admissible complex structure, thus for any fiber $F \subset P_\psi$ of the
mapping torus, $\{\text{const}\} \times F \subset \RR \times P_\psi$ 
is an embedded holomorphic
curve for any~$J_0$.  The task is therefore to find a foliation by
holomorphic curves in $\RR \times (S^1 \times \DD)$, which have a
puncture asymptotic to the orbit at $\{\rho = 0\}$ and fit together
smoothly with the fibers $\{\text{const}\} \times F$.  For a sufficiently
symmetric choice of the data, this is merely a matter of writing down the
Cauchy-Riemann equations and solving them: we have $\lambda_0 = 
f\ d\theta + g\ d\phi$ and $X_0 = \frac{g'}{D} \p_\theta - \frac{f'}{D}
\p_\phi$, which reduce to $d\phi$ and $\p_\phi$ respectively near
$\p(S^1 \times \DD)$.  Choose vector fields 
$$
v_1 = \p_\rho,
\qquad
v_2 = -g(\rho) \p_\theta + f(\rho) \p_\phi
$$
to span $\xi_0$, along with a smooth function $\beta(\rho) > 0$, and define $J_0 \in \jJ(\hH_0)$
at $(\theta,\rho,\phi) \in S^1 \times \DD$ by the condition
$J_0 v_1 = \beta(\rho) v_2$.
The function $\beta$ can be chosen so that this definition of $J_0$
extends smoothly to $\rho = 0$, and we shall assume $\beta(\rho) = 1$
for $\rho$ outside a neighborhood of~$0$.
Then one can compute (cf.~\cite{Wendl:OTfol}*{\S 4.2}) that in
conformal coordinates $(s,t)$, a map
$$
u(s,t) = (a(s,t),\theta(s,t),\rho(s,t),\phi(s,t))
$$ 
with $\rho(s,t) < 1 - \delta$ is $J_0$--holomorphic if and only 
if it satisfies the equations
\begin{equation*}
%\label{eqn:CR}
\begin{aligned}
a_s &= f\theta_t + g\phi_t \quad\quad\quad\quad &
\rho_s &= \frac{1}{\beta D} (f'\theta_t + g'\phi_t) \\
a_t &= -f\theta_s - g\phi_s &
\rho_t &= -\frac{1}{\beta D} (f'\theta_s + g'\phi_s) \\
\end{aligned}
\end{equation*}
where $f$, $g$, $D$ and $\beta$ are all functions of $\rho(s,t)$.
If $\rho(s,t) \ge 1 - \delta'$, then $g'(\rho) = 0$ and $g(\rho) = 1 =
\beta(\rho)$, so the two equations on the right become
$$
\rho_s = - \theta_t,
\qquad
\rho_t = \theta_s.
$$
There are then solutions of the form
$$
u : [0,\infty) \times S^1 \to \RR \times (S^1 \times \DD) : 
(s,t) \mapsto (a(s),t,\rho(s),\phi_0)
$$
for any constant $\phi_0 \in S^1$, where $a(s)$ and $\rho(s)$ solve the
ordinary differential equations
\begin{equation}
\label{eqn:ODE}
\frac{da}{ds} = f(\rho),
\qquad
\frac{d\rho}{ds} =
\begin{cases}
\frac{f'(\rho)}{\beta(\rho) D(\rho)} & \text{ if $\rho < 1 - \delta$,}\\
-1                           & \text{ if $\rho \ge 1 - \delta'$.}
\end{cases}
\end{equation}
Let us now add the following standing assumptions for $f$ and~$g$:
\begin{enumerate}
\item
$f'(\rho) < 0$ for all $\rho \in (0,1-\delta)$.
\item
$f'(\rho) / g'(\rho)$ is a constant irrational number close to zero
for sufficiently small~$\rho > 0$.
\end{enumerate}
Observe that one can impose these conditions and in addition require
$f(0) > 0$ to be arbitrarily small, the latter being the period
of the Reeb orbit at $\rho = 0$ (see Remark~\ref{remark:smallPeriods}).
Now by a straightforward computation of the linearized Reeb flow,
the second assumption ensures that this orbit and all its
multiple covers will be nondegenerate, and these are the only periodic
orbits in some neighborhood.  By the first assumption,
the unique solution to
\eqref{eqn:ODE} with $\rho(0) = 1$ and any given value of $a(0) \in \RR$
yields a $J_0$--holomorphic half-cylinder $u$ which is 
positively asymptotic
as $s \to \infty$ to the embedded orbit at $\{\rho = 0 \}$
and has $a(s,t)$ and $\phi(s,t)$ both constant near the boundary.  The
image can therefore be attached smoothly to the holomorphic
fiber $\{\text{const}\} \times F \subset \RR \times P_\psi$, giving
an extension of the latter to an embedded $J_0$--holomorphic curve 
in $\RR\times M$, with no boundary and with punctures asymptotic to the
binding orbits of the open book; the collection of all these curves
defines the $\RR$--invariant foliation of $\RR \times M$ shown in
Figure~\ref{fig:J0holomorphic}.  This proves Proposition~\ref{prop:J0openbook}
except for the index calculation (see Equation~\eqref{eqn:index} below,
and the ensuing discussion).

\begin{figure}
\includegraphics{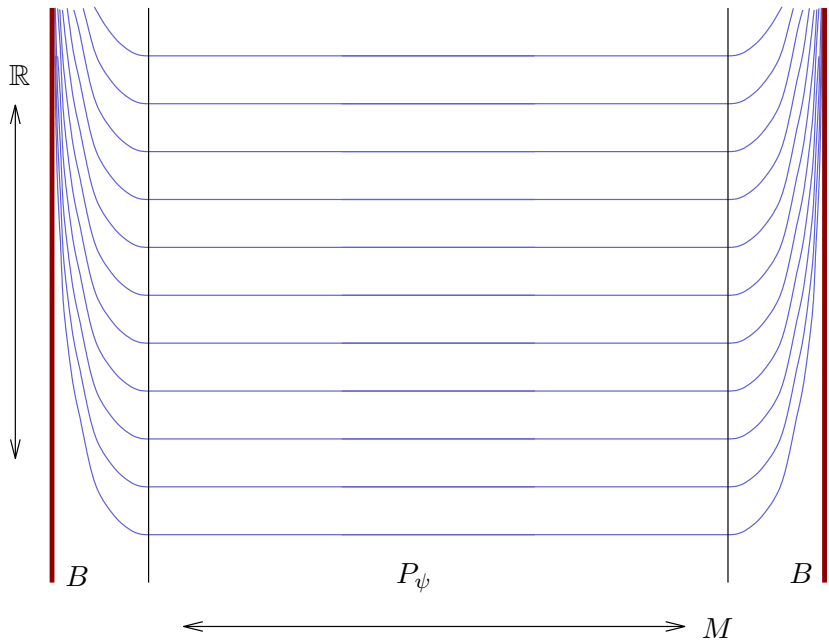}
\caption{\label{fig:J0holomorphic} 
The $J_0$--holomorphic curves that foliate $\RR\times M$ and project to
any given open book decomposition of~$M$.  The $\RR$--component of each
curve is constant for the part within the mapping torus $P_\psi$, and
approaches $+\infty$ near the binding.}
\end{figure}

Having constructed a foliation by $J_0$--holomorphic curves for the data
$\hH_0 = (\xi_0,X_0,\omega_0)$, we now wish to deform the foliation as
$\hH_0$ is perturbed to the contact data
$\hH_\epsilon = (\xi_\epsilon,X_\epsilon,\omega_\epsilon)$ for sufficiently
small~$\epsilon$.  The crucial consequence of Proposition~\ref{prop:SHS} is 
that we can pick an
almost complex structure $J_\epsilon \in \jJ(\hH_\epsilon)$ that is
$C^\infty$--close to $J_0$.  By the construction of $X_\epsilon$ and
$\xi_\epsilon$, we can also assume $J_\epsilon$ equals $J_0$ on a neighorhood
of $\RR \times B \subset \RR\times M$,
thus the perturbation of our foliation will be a straightforward application
of the implicit function theorem for holomorphic curves.  This is however
the point where we'll need the open book to be \emph{planar}, as otherwise
the virtual dimension of the moduli space we've constructed turns out to
be too small.

Let us compute this dimension.
Our assumptions on $f$ and $g$ for $\rho$ close to zero imply that each of
the binding orbits $\gamma \subset B$ has Conley-Zehnder index
$\muCZ(\gamma) = 1$ with respect to the natural trivialization defined by
the coordinates.  Let $u : \dot{\Sigma} \to \RR \times M$ denote one
of the holomorphic curves in our foliation.  Since it is embedded, it sits
in a moduli space whose virtual dimension (the \emph{index} of $u$) equals
the Fredholm index of the linearized normal Cauchy-Riemann operator,
cf.~\cite{Wendl:automatic}.  This index is
\begin{equation}
\label{eqn:index}
\ind(u) = \chi(\dot{\Sigma}) + 2 c_1(N_u) + \sum_{\gamma \subset B} 
\muCZ(\gamma),
\end{equation}
where $c_1(N_u)$ denotes the
\emph{relative} first Chern number of the normal bundle of $u$ with respect
to the natural trivializations of $\xi_0$ (which equals $N_u$ at the asymptotic
limits) defined by our coordinates 
near~$B$.  Assume the pages (and thus also $\dot{\Sigma}$) have genus~$g$.
We claim now that $c_1(N_u) = 0$, and thus \eqref{eqn:index} implies 
$$
\ind(u) = 2 - 2g.
$$
Indeed, since $u$ is always 
transverse to the subspaces spanned by $\p_a$ and $X_0$,
one can define the normal bundle to consist of these spaces in $P_\psi$,
and extend it into the neighborhood of $B$ so that
it always contains $\p_\phi$: there is thus a nonzero section of $N_u$
which looks like $X_0$ over $P_\psi$ and $\p_\phi$ near the ends, and
the latter is constant in the asymptotic trivialization.

Assume from now on that $g=0$, so the $J_0$--holomorphic curves
constructed above have index~$2$.  We will apply the following strong
version of the implicit function theorem, which is valid only for a special
class of punctured holomorphic spheres in dimension four; proofs of the
following (in slightly more general versions) may be found in
\cites{Wendl:thesis,Wendl:BP1}, and a special
case appeared already in \cite{HWZ:props3}.

\begin{prop}[Strong implicit function theorem]
\label{prop:IFT}
Assume $M$ is any closed $3$--manifold with stable Hamiltonian structure
$\hH = (\xi,X,\omega)$, $J \in \jJ(\hH)$, and
$$
u = \left(u^\RR,u^M\right) : \dot{\Sigma} \to\RR\times M
$$
is a punctured $J$--holomorphic curve with the following properties:
\begin{enumerate}
\item $u$ is embedded, and asymptotic to distinct simply covered periodic
orbits at each puncture.
\item $\dot{\Sigma}$ has genus zero.
\item All asymptotic orbits are nondegenerate with odd Conley-Zehnder index.
\item $\ind(u) = 2$.
\end{enumerate}
Then $u$ is Fredholm regular and belongs to a
smooth $2$--parameter family of embedded curves
$$
u_{(\sigma,\tau)} = \left(u_\tau^\RR + \sigma,u_\tau^M\right) : 
\dot{\Sigma} \to \RR\times M,
\qquad (\sigma,\tau) \in \RR\times (-1,1)
$$
with $u_{(0,0)} = u$,
whose images foliate an open neighborhood of $u(\dot{\Sigma})$
in $\RR \times M$.  Moreover, the maps $u_\tau^M : \dot{\Sigma} \to M$
are all embedded and foliate an open neighborhood of $u^M(\dot{\Sigma})$
in~$M$.
\end{prop}
Here, \emph{Fredholm regular}\footnote{The term \emph{unobstructed} also appears
often in the literature as a synonym.} 
means the moduli space of (unparametrized) 
$J$--holomorphic curves
near $u$ can be described as the zero set of a nonlinear
Cauchy-Riemann type operator whose linearization at $u$ is a 
surjective Fredholm operator.
The usual implicit function theorem in a Banach manifold setting then implies
that this moduli space is a smooth $2$--manifold near $u$, thus the
$2$--parameter family obtained in the proposition is unique up to changes
of parametrization.  The reason for the nice geometric structure of the
family is that if $u$ is embedded, then all nearby curves can be
described via sections of the normal bundle $N_u$ which must
satisfy a linear Cauchy-Riemann type equation.  Since $\dim M = 3$, $N_u$
is a complex line bundle, so the zeroes of its sections (with prescribed
asymptotic behavior) can
be counted and related to the same homotopy invariant quantities that
figure into the index formula.  Notably, the integer $c_1(N_{u})$,
defined as a relative Chern number with respect
to certain special asymptotic trivializations
(cf.~\cite{Wendl:automatic}), satisfies the relation
$$
2 c_1(N_u) = \ind(u) - 2 + 2g + \#\Gamma_0,
$$
where $g$ is the genus of $\dot{\Sigma}$ and $\Gamma_0$ is the set of
punctures at which the asymptotic orbit has even Conley-Zehnder index.
Thus in the present case, this number vanishes and implies that nontrivial
sections satisfying the relevant Cauchy-Riemann type equation must be
zero free.  It follows that this solution set can have dimension at most~$2$,
so the linearized operator has no codimension, and all nearby curves are
push-offs of zero free sections and hence disjoint.  The corresponding
statement about the projected maps $u_\tau^M : \dot{\Sigma} \to M$ follows
easily from this due to $\RR$--invariance.  We should note one
more important property of Fredholm regularity which will be useful presently: 
it allows one to apply the
implicit function theorem to deform the family $u_{(\sigma,\tau)}$
under small perturbations of~$J$.

We apply the above machinery as follows.
Let $\mM_0$ denote the connected $2$--dimensional moduli space of curves
that form the $J_0$--ho\-lo\-mor\-phic open book; dividing by the
natural $\RR$--action, we have $\mM_0 / \RR \cong S^1$.
For some small $\epsilon_0 > 0$, assume $\{J_\tau\}_{\tau \in 
(-\epsilon_0,\epsilon_0)}$ is a smooth family of almost complex
structures such that $J_\tau = J_0$ for $\tau \le 0$,
$J_\tau \in \jJ(\hH_\tau)$ for $\tau > 0$ and $J_\tau \equiv J_0$ in a 
neighborhood of $\RR \times B$.  Now define the moduli space
$$
\widehat{\mM} = \{ (\tau,u)\ |\ \tau \in (-\epsilon_0,\epsilon_0),
\ \text{$u$ is a finite energy $J_\tau$--holomorphic curve} \},
$$
let $\mM$ denote the connected component of $\widehat{\mM}$ containing
$\{0\} \times \mM_0$, and for each $\tau \in (-\epsilon_0,\epsilon_0)$ 
define the subset 
$\mM_\tau := \{ (\tau,u) \in \mM \}$.  An argument by positivity
of intersections as in \cites{HWZ:tight3sphere,ACH} shows that all
curves $u \in \mM_\tau$ are embedded, and no two curves in $\mM_\tau$
(for fixed $\tau$) can intersect.  Moreover, they all have index~$2$
and genus zero, and have asymptotic orbits with exclusively \emph{odd}
Conley-Zehnder index.  Proposition~\ref{prop:IFT} now implies that
each $\mM_\tau$
is a smooth $2$--manifold and $\mM_\tau / \RR$ is a $1$--manifold that
locally foliates $M \setminus B$.
It also follows that $\mM$ is a smooth $3$--manifold.

We claim that for $\tau > 0$
sufficiently small, $\mM_\tau / \RR$ is diffeomorphic to
$S^1$ and hits every point in $M \setminus B$.
To see this, pick a loop $\ell \subset M \setminus B$ that passes
once transversely through the projection of every curve in $\mM_0$.
This defines an evaluation map
$$
\ev : \mM_0 / \RR \to \ell,
$$
which is a diffeomorphism.  Moreover since every curve in $\mM$ can be
compactified to a surface with boundary in~$B$, the algebraic intersection
number of $u \in \mM$ with $\ell$ is invariant, thus
every curve in $\mM$
must intersect~$\ell$, and the aforementioned positivity of intersections
argument implies that for any given $p \in \ell$ and $\tau \in 
(-\epsilon_0,\epsilon_0)$, there is at most one element of $\mM_\tau / \RR$
with $p$ in its image.  Now for every $u \in \mM_0 / \RR$,
pick an open neighborhood $(0,u) \in \uU_u \subset \mM / \RR$ 
sufficiently small
so that all curves in $\uU_u$ hit $\ell$ transversely, exactly once.
The evaluation map $\ev : \uU_u \to \ell$ is therefore well defined,
and writing the projection $\pi_1 : \mM / \RR \to (-\epsilon_0,\epsilon_0) :
(\tau,u) \mapsto \tau$, the map
$$
\pi_1 \times \ev : \uU_u \to (-\epsilon_0,\epsilon_0) \times \ell
$$
has nonsingular derivative at $(0,u)$, hence we can find an open neighborhood
$\ell_u \subset \ell$ of $\ev(u)$ and a number $\epsilon_u \in (0,\epsilon_0)$ such
that $(-\epsilon_u,\epsilon_u) \times \ell_u$ is in the image of $\uU_u$ under
$\pi_1 \times \ev$.
The sets $\ell_u$ for all $u \in \mM_0 / \RR$ now form an open covering of
the compact set $\ell$, so we can pick finitely many curves
$u_1,\ldots,u_N \in \mM_0$ such that $\ell = \ell_{u_1} \cup \ldots \cup \ell_{u_N}$,
and set $\epsilon := \min_j \{ \epsilon_{u_j} \} > 0$.  Now the image
of $\uU_{u_1} \cup \ldots \cup \uU_{u_N}$ under $\pi_1 \times \ev$
contains $(-\epsilon,\epsilon) \times \ell$, hence for any 
$\tau \in (0,\epsilon)$ and $p \in \ell$, there is a curve in $\mM_\tau / \RR$
passing transversely through $\ell$ at~$p$.  It follows that the
evaluation map extends to $\mM_\tau / \RR$ as a diffeomorphism
$\ev : \mM_\tau / \RR \to \ell$.
Finally, we observe from Proposition~\ref{prop:IFT} that the set
$$
\{ p \in M \setminus B \ |\ 
\text{$p$ is in the image of some $u \in \mM_\tau / \RR$} \}
$$
is open, and it is also closed since $\mM_\tau / \RR$ is compact, thus
it is all of $M\setminus B$.

One can perturb $\lambda_\epsilon$ further
so that it becomes nondegenerate, and repeating
the argument above then gives the desired foliation by holomorphic curves 
for a nondegenerate contact form.

\end{document}